\documentclass[12pt]{article}

\usepackage{color}
\usepackage{amscd}
\usepackage{bm}
\advance \textwidth -60truemm\relax

\advance\oddsidemargin -1truein\relax

\evensidemargin=\oddsidemargin

\advance\textheight -70truemm\relax
\advance\topmargin -1truein\relax
\unitlength\textwidth
\divide\unitlength by 150\relax

\usepackage{amsmath,amssymb}
\usepackage{mathptmx}
\usepackage{bm}
\bmdefine{\aaa}{a}
\bmdefine{\bbb}{b}
\bmdefine{\ccc}{c}
\bmdefine{\ddd}{d}
\bmdefine{\sss}{s}
\bmdefine{\uuu}{u}
\bmdefine{\vvv}{v}
\bmdefine{\www}{w}
\bmdefine{\eee}{e}
\bmdefine{\xxx}{x}
\bmdefine{\yyy}{y}
\bmdefine{\zzz}{z}
\bmdefine{\zerovec}{0}

\newcommand{\XXX}{\mathbb{X}}
\newcommand{\CCC}{\mathbb{C}}
\newcommand{\RRR}{\mathbb{R}}

\newcommand{\qed}{\nolinebreak\rule{.3em}{.6em}}
\newcommand{\rank}{\mathrm{rank}}
\newcommand{\mtrank}{\mathrm{mtrank}\,}
\newcommand{\grank}{\mathrm{grank}\,}

\newcommand{\typicalrankR}{{\mathrm{typical\_rank_\RRR}}}

\newcounter{thmno}
\numberwithin{thmno}{section}
\newtheorem{thm}[thmno]{Theorem}

\newtheorem{prop}[thmno]{Proposition}

\newtheorem{notation}[thmno]{Notation}

\def\fl{\text{fl}}

\begin{document}

\title{Upper bound of typical ranks of $m\times n\times ((m-1)n-1)$ tensors over the real number field}
\author{Toshio Sumi, Toshio Sakata and Mitsuhiro Miyazaki}

\maketitle

\begin{abstract}
Let $3\leq m\leq n$.
We study typical ranks of $m\times n\times ((m-1)n-1)$ tensors
over the real number field.
The number $(m-1)n-1$ is a minimal typical rank of $m\times n\times ((m-1)n-1)$ tensors over the real number field.
We show that a typical rank of $m\times n\times ((m-1)n-1)$ tensors
over the real number field is less than or equal to $(m-1)n$ and
in particular, $m\times n\times ((m-1)n-1)$ tensors
over the real number field has two typical ranks $(m-1)n-1, (m-1)n$
if $m\leq \rho(n)$, where $\rho$ is the Hurwitz-Radon function defined as
$\rho(n)=2^b+8c$ for nonnegative integers $a,b,c$ 
such that $n=(2a+1)2^{b+4c}$ and $0\leq b<4$.
\end{abstract}

\section{Introduction}

Kolda and Bader \cite{Kolda-Bader:2009} introduced many applications 
of tensor decomposition analysis 
in various fields such as signal processing, computer vision, data mining,
and others.
Tensor decomposition concerns with its rank and approximation of tensor
decomposition concerns with typical ranks.
In this paper we discuss the typical rank for $3$-way arrays ($3$-tensors).
A $3$-way array 
$$(a_{ijk})_{1\leq i\leq m,\ 1\leq j\leq n,\ 1\leq k\leq p}$$
with size $(m,n,p)$ is called an $m\times n\times p$ tensor.
An $m\times n\times p$ tensor of form 
$$(x_iy_jz_k)_{1\leq i\leq m,\ 1\leq j\leq n,\ 1\leq k\leq p}$$
is called a rank one tensor.
A rank of a tensor $T$, denoted by $\rank\, T$, 
is defined as the minimal number of rank one tensors which describe $T$ as a sum.
The rank depends on the base field.  
\par

Throughout this paper, we assume that the base field is the real number field $\RRR$.
Let $\mathbb{R}^{m\times n\times p}$ be the set of $m\times n\times p$ tensors with Euclidean topology.
A number $r$ is a typical rank of $m\times n\times p$ tensors
if the set of tensors with rank $r$ contains 
a nonempty open semi-algebraic set
of $\mathbb{R}^{m\times n\times p}$
 (see \cite{Friedland:2008}).
We denote by $\typicalrankR(m,n,p)$ the set of typical ranks of
$\mathbb{R}^{m\times n\times p}$.
Note that
$$\typicalrankR(m_1,m_2,m_3)=\typicalrankR(m_i,m_j,m_k)$$
for any $i,j,k$ with $\{i,j,k\}=\{1,2,3\}$.
If $s$ and $t$ are typical ranks of tensors with $s\leq t$, then
$u$ is also a typical rank of tensors for any $u$ with $s\leq u\leq t$. 
The minimal number of $\typicalrankR(m,n,p)$ is equal to
the generic rank $\grank(m,n,p)$ of the set of $m\times n\times p$ tensors over the complex number field \cite{Friedland:2008}.
we denote by $\mtrank(m,n,p)$ by the maximal typical rank of $\mathbb{R}^{m\times n\times p}$.  Then
$$\typicalrankR(m,n,p)=[\grank(m,n,p),\mtrank(m,n,p)]\cap \mathbb{Z}.$$
However, it is only known that one or two typical ranks of tensors.

For $m=1$, the rank of $1\times n\times p$ tensor is its matrix rank and therefore the set of typical ranks of $1\times n\times p$ tensors
consists of one number $\min(n,p)$.
In the case where $m=2$, the set of typical ranks of $2\times n\times p$ tensors is well-known \cite{tenBerge-etal:1999}:
$$\typicalrankR(2,n,p)=\begin{cases}
\{p\}, & n<p\leq 2n \\
\{2n\}, & 2n<p \\
\{p,p+1\}, & n=p\geq 2
\end{cases}
$$
Suppose that $3\leq m\leq n$.
The typical rank of $\RRR^{m\times n\times p}$ is quite different
from that of $\RRR^{2\times n\times p}$.
Let $\rho(n)$ be the Hurwitz-Radon number, that is,
$\rho(n)=2^b+8c$ for nonnegative integers $a,b,c$ 
such that $n=(2a+1)2^{b+4c}$ and $0\leq b<4$.
If $p>(m-1)n$ then the set of typical ranks of $m\times n\times p$ tensors
is just $\{\min(p,mn)\}$.
For $p=(m-1)n$, the set of typical ranks of $m\times n\times p$ tensor 
is $\{p\}$ (resp. $\{p,p+1\}$) if and only if $m>\rho(n)$ (resp. $m\leq \rho(n)$).

The purpose of this paper is to give an upper bound of typical ranks
of $m\times n\times ((m-1)n-1)$ tensors.

\begin{thm} \label{thm:main}
Let $3\leq m\leq n$ and $p=(m-1)n-1$.
A typical rank of $m\times n\times p$ tensors is less than or equal to $p+1$.
In particular, $\typicalrankR(m,n,p)$ is a subset of $\{p,p+1\}$.
\end{thm}

By \cite[Theorem~1.1]{Miyazaki-etal:2012a} we directly have the following
proposition.

\begin{prop}
Let $m\leq n$.
Suppose that $3\leq m\leq \rho(n)$ or 
that both $m$ and $n$ are congruent to $3$ modulo $4$.
Then 
$$\typicalrankR(m,n,(m-1)n-1)=\{(m-1)n-1,(m-1)n\}.$$
\end{prop}


\section{A proof}

In this section, we show the proof of Theorem~\ref{thm:main}.

First, we establish terminology.

\begin{notation}
\begin{enumerate}
\item For an $m\times n$ matrix $M$, we denote the $i\times j$ matrix consisting of the first $i$ rows and the first $j$ columns of $M$ by $M^{\leq i}_{\leq j}$.
\item For an $m\times n$ matrix $M$, we denote the $j$-th column vector of $M$ by $M_{=j}$.
\item For a square matrix $M$, we denote the determinant of $M$ by $|M|$.
\item For a tensor $T = (t_{ijk})_{1\leq i\leq m,\ 1\leq j\leq n,\ 1\leq k\leq p}$, we denote it by $(T_1; \ldots;T_p)$,
where $T_k = (t_{ijk})_{1\leq i\leq m,\ 1\leq j\leq n}$ for $k = 1,\ldots,p$ is an $m\times n$ matrix.
\item For a vector $\ccc=(c_1,\ldots,c_n)^\top\in\RRR^n$, we denote the Euclidean norm of $\ccc$ by $||\ccc||$, that is, $||\ccc||=\sqrt{\sum_{k=1}^n c_k^2}$.
\end{enumerate}
\end{notation}

\begin{prop} \label{prop:coverrank}
Let $\mathcal{R}(m,n,p;r):=\{T\in\RRR^{m\times n\times p}\mid \rank\;T\leq r\}$
and $\tau$
be a canonical map from $\RRR^{m\times n\times p}$ onto $\RRR^{m\times n\times (p-1)}$ which sends $(Y_1;\ldots;Y_{p-1};Y_p)$ to $(Y_1;\ldots;Y_{p-1})$.
If the set $\tau(\mathcal{R}(m,n,p;r))$ is a dense subset of $\RRR^{m\times n\times (p-1)}$, then $\mtrank(m,n,p-1)\leq r$.
\end{prop}

\begin{proof}
Since $\tau(\mathcal{R}(m,n,p;r))$ is a dense, semi-algebraic set, 
its interior is an open, dense semi-algebraic set and thus a Zariski open set.
Therefore, $\mtrank(m,n,p-1)\leq r$ follows.
\qed
\end{proof}

Let $3\leq m\leq n$, $p_0=(m-1)n$, and $p=p_0-1$.
We want to show that $\mtrank(m,n,p)\leq p_0$.
To do this, we show that there are a dense subset $U$ of 
$\RRR^{m\times n\times p}$ and a section $s\colon U
\to \RRR^{m\times n\times p_0}$ such that $\rank\,s(T)\leq p_0$ for any
tensor $T$ of $U$.
Then, by Proposition~\ref{prop:coverrank}, we conclude that 
$\mtrank(m,n,p)\leq p_0$.

Let $W$ be the set consisting of $\begin{pmatrix} A_1&A_2\\ A_3&A_4\end{pmatrix}$
such that $A_1$ is an $(n-1)\times (n-1)$ matrix and has distinct eigenvalues,
$P$ is a nonsingular matrix so that $P^{-1}A_1P$ is a diagonal matrix,
and each element of $P^{-1}A_2$ is nonzero.

Let $\fl_1\colon \RRR^{m_1\times m_2\times m_3} \to \RRR^{m_1m_3\times m_2}$
be a bijection defined as
$$(A_1;A_2;\ldots;A_{m_3}) \mapsto \begin{pmatrix} A_1\\ A_2\\ \vdots\\
A_{m_3}\end{pmatrix}$$
and $\fl_2\colon \RRR^{m_1\times m_2\times m_3} \to \RRR^{m_1\times m_2m_3}$
be a bijection defined as
$$(A_1;A_2;\ldots;A_{m_3}) \mapsto (A_1,A_2,\ldots,A_{m_3}).$$

For $(X_1;\ldots;X_m)\in\RRR^{n\times p\times m}$, we put
$\begin{pmatrix} A\\ \bbb^\top\end{pmatrix}=\begin{pmatrix} X_1\\ \vdots\\ X_{m-1}\end{pmatrix}$,
and 
$$(Z_1;\ldots;Z_{m-1})=\fl_2^{-1}(X_mA^{-1},\zerovec)\in\RRR^{n\times p_0\times (m-1)}.$$
Now suppose that $m>3$.
Let $\mathfrak{T}$ be the subset of $\RRR^{n\times p\times m}$ consisting of
$(X_1;\ldots;X_m)$ satisfying the following conditions:
\begin{eqnarray}
|A|\ne 0. \label{cd1} \\ 
|(Z_{m-1})^{\leq n-1}_{\leq n-1}|\ne 0. \label{cd2} \\
\text{All eigenvalues of $(Z_{m-1})^{\leq n-1}_{\leq n-1}$ are distinct.} \label{cd3} \\
Z_k \in W \text{ for } 1\leq k\leq m-2. \label{cd4} \\
\left|\sum_{k=1}^{m-2} x_kZ_k-x_mE_n\right| \text{ is irreducible.} \label{cd5} 
\end{eqnarray}

If $\left|\sum_{k=1}^{m-2} x_kZ_k-x_mE_n\right|$ is irreducible, then so is
$\left|\sum_{k=1}^{m-1} x_kZ_k-x_mE_n\right|$ for any $Z_{m-1}$.
Since $m>3$, the set $\mathfrak{T}$ is a nonempty Zariski open set.


We consider the following two maps:
$$
\begin{array}{ll}
f\colon \mathfrak{V}_1 \to \RRR^{n\times n\times (m-1)}; & 
f(Y_1;\ldots;Y_m)=Y_m(\fl_1(Y_1;\ldots;Y_{m-1}))^{-1},\\[1mm]
g\colon \RRR^{n\times p\times m} \to \RRR^{n\times p_0\times m}; &  
g(X_1;\ldots;X_m)=\fl_1^{-1}\begin{pmatrix} \begin{matrix} A\\ \bbb^\top\end{matrix} & \begin{matrix} \zerovec\\ 1\end{matrix} \\
X_m & \zerovec\end{pmatrix},
\end{array}
$$
where $$\mathfrak{V}_1:=\{(Y_1;\ldots;Y_m)\in \RRR^{n\times p_0\times m} \mid |\fl_1(Y_1;\ldots;Y_{m-1})|\ne 0\}.$$
Then
$$f\circ g(X_1;\ldots;X_m)=(X_mA^{-1},\zerovec)$$
and more generally 
$$f(\fl_1^{-1}\begin{pmatrix} \begin{matrix} A\\ \bbb^\top\end{matrix} & \begin{matrix} \zerovec \\ 1\end{matrix}
\\ X_m & \ccc \end{pmatrix})
=((X_m-\ccc \bbb^\top)A^{-1},\ccc)$$
for $(X_1;\ldots;X_m)\in g^{-1}(\mathfrak{V}_1)$.

Now, we fix $(X_1;\ldots;X_{m-1})\in \mathfrak{T}$.
Putting $(Z_1;\ldots;Z_{m-1})=(X_mA^{-1},\zerovec)$,
conditions \eqref{cd2}--\eqref{cd5} hold.
Since the characteristic polynomial $|Z_{m-1}-\lambda E_n|$ is divisible by $\lambda$ but not by $\lambda^2$,
we have $(Z_1;\ldots;Z_{m-1})\in\mathfrak{C}$,
where
$$\mathfrak{C}=\{(Y_1;\ldots;Y_m) \in \mathbb{R}^{n\times n\times(m-1)}\mid\;
|\sum_{k=1}^{m-1} a_kY_k-a_mE_n|<0
\text{ for some $(a_1,\ldots,a_m)^\top\in \mathbb{R}^m$}\}.$$
Therefore, $f\circ g(\mathfrak{T})\subset \mathfrak{C}$.
In the previous paper \cite{Sumi-etal:2011a}, 
we showed that $\rank\,X=p_0$ for any $X\in\mathfrak{V}_1$ with
$f(X)\in \mathfrak{C}\cap\mathfrak{W}_2$,
where
\begin{equation*}
\begin{split}
\mathfrak{W}_2:=\{Z=(Z_1&;\ldots;Z_{m-1}) \in \RRR^{n\times n\times(m-1)}\mid \\
&  Z_k\in W \text{ for each $1\leq k\leq m-1$},\ \text{$|\sum_{k=1}^{m-1} x_kZ_k-x_mE_n|$ is irreducible.}\}.
\end{split}
\end{equation*}

For any $\ccc\in\RRR^n$ with sufficiently small $||\ccc||$ and 
$(Z_1;\ldots;Z_{m-1})=((X_m-\ccc \bbb^\top)A^{-1},\ccc)$,
the conditions \eqref{cd1}--\eqref{cd5} hold.
In addition, since $\mathfrak{C}$ is open, there exists $\ccc$ such that
$Z_{m-1}\in W$ and $(Z_1;\ldots;Z_{m-1})\in \mathfrak{C}\cap\mathfrak{W}_2$.
Therefore, we have
$$\rank (X_1;\ldots;X_{m-1})\leq \rank\; \fl_1^{-1}(\begin{pmatrix} \begin{matrix} A\\ \bbb^\top\end{matrix} & \begin{matrix} \zerovec \\ 1\end{matrix}
\\ X_m & \ccc \end{pmatrix})=p_0.$$
We complete the proof of the following theorem.

\begin{thm} \label{thm:m>3}
Let $4\leq m\leq n$.
$$\typicalrankR(m,n,(m-1)n-1)=\{(m-1)n-1\} \text{ or } \{(m-1)n-1,(m-1)n\}.$$
\end{thm}

In the case when $m=3$, the condition \eqref{cd5} must be replaced as
the condition that 
\begin{equation} \label{eq:char.poly.}
\left|\sum_{k=1}^{m-1} x_kZ_k-x_mE_n\right| \text{ is irreducible.}
\end{equation}
We show that the replacement is possible.

\begin{thm} \label{thm:m=3}
Let $3=m\leq n$.
$$\typicalrankR(m,n,(m-1)n-1)=\{(m-1)n-1\} \text{ or } \{(m-1)n-1,(m-1)n\}.$$
\end{thm}

\begin{proof}
Note that the $n$-th column $(Z_{m-1})_{=n}$ of $Z_{m-1}$ is zero.
Then $H^{-1}Z_{m-1}H$ is equal to
$$Y_2=
\begin{pmatrix}0 & 0&\cdots & 0\\
-1 & 0 & \cdots & v_{2}\\
\vdots & \ddots & \ddots & \vdots \\
0& \cdots & -1 & v_n
\end{pmatrix}
\}$$
for some nonsingular matrix $H$.
Let $P$ be a real vector space of dimension $n(n+3)/2-1$ with basis
$$\{x_1^ax_2^bx_3^c\mid 0\leq a,b,c\leq n, a+b+c=n, b,c\ne n\}$$
and $S$ be the set defined as
$$S:=\{(Y_1;Y_2) \in \RRR^{n\times n\times 2} \mid
Y_1=
\begin{pmatrix} u_{11} & 0&\cdots & u_{11}\\
u_{21} & u_{22} & \ddots & \vdots \\
\vdots & \ddots & \ddots & 0\\
u_{n1}& \cdots & u_{n-1,1} & u_{n1}
\end{pmatrix}
\}$$
which is isomorphic to a vector space of dimension $n(n+3)/2-1$.
Let $g$ be a map from $S$ to $P$
defined as
$$g((Y_1;Y_2))=|x_1Y_1+x_2Y_2+x_3E_n|-x_3^n.$$
Note that the polynomial \eqref{eq:char.poly.} is irreducible if and only if $G((Z_1;Z_2))$
is irreducible.
Now we show that the Jacobian of $G$ is nonzero.
\par
Suppose that for constants $c(v_j)$, $c(u_{ij})$, the linear equation
\begin{equation}\label{eqn:1}
\sum_{j=2}^n c(v_j)\frac{\partial g}{\partial v_j}
+\sum_{1\leq j\leq i\leq n} c(u_{ij})\frac{\partial g}{\partial u_{ij}}=0
\end{equation}
holds.
We show that all of $c(v_j)$, $c(u_{ij})$ are zero by induction on $n$.
It is easy to see that the assertion holds in the case where $n=1$.
As the induction assumption, we assume that the
assertion holds in the case where $n-1$ instead of $n$.
We put
$$\lambda_j=u_{jj}x_1+x_3 \text{ and  } \mu(a,b)=\prod_{t=a}^{b}\lambda_t.$$
After a partial derivation, we put $u_{ij}=0$ ($i>j$) and then
have the following equations:

$$
\begin{array}{lcll}
\displaystyle\frac{\partial g}{\partial v_j}&=& 
x_2^{n-j+1}\mu(1,j-1)
& (2\leq j\leq n)\\ 
\displaystyle\frac{\partial g}{\partial u_{11}}&=& 
x_1
 \left| \begin{matrix} \lambda_{2} && & & (v_{2}+1)x_2\\
-x_2& \lambda_{3} &&& v_{3}x_2\\
& \ddots& \ddots&&\vdots \\
& & -x_2 & \lambda_{n-1}& v_{n-1}x_2\\
& && -x_2 & \lambda_{n}+v_nx_2\\
\end{matrix}\right|
& \\ 
\displaystyle\frac{\partial g}{\partial u_{jj}}&=& 
x_1\mu(1,j-1)
 \left| \begin{matrix} \lambda_{j+1} && & & v_{j+1}x_2\\
-x_2& \lambda_{j+2} &&& v_{j+2}x_2\\
& \ddots& \ddots&&\vdots \\
& & -x_2 & \lambda_{n-1}& v_{n-1}x_2\\
& && -x_2 & \lambda_{n}+v_nx_2\\
\end{matrix}\right|
& (2\leq j\leq n)\\ 
\displaystyle\frac{\partial g}{\partial u_{ij}}&=& 
\displaystyle -x_1x_2^{n-i}\mu(j+1,i-1)
\left| \begin{matrix} \lambda_{1} & & && u_{11}x_1\\
-x_2& \lambda_{2} &&& v_{2}x_2\\
& \ddots& \ddots&& \vdots \\
& & -x_2 & \lambda_{j-1}& v_{j-1}x_2\\
& && -x_2 & v_jx_2\\
\end{matrix}\right|
& (1\leq j<i\leq n) \\ 
\end{array}
$$

\noindent
By seeing terms divisible by $\lambda_1$ in the left hand side of
\eqref{eqn:1}, we have
\begin{equation*}\label{eqn:3}
\sum_{j=2}^n c(v_j)\frac{\partial g}{\partial v_j}
+\sum_{2\leq j\leq i\leq n} c(u_{ij})h_{ij}=0
\end{equation*}
where 
$$h_{ij}=
-x_1x_2^{n-i}\mu(j+1,i-1)
\left| \begin{matrix} \lambda_{1} & & && 0\\
-x_2& \lambda_{2} &&& v_{2}x_2\\
& \ddots& \ddots&& \vdots \\
& & -x_2 & \lambda_{j-1}& v_{j-1}x_2\\
& && -x_2 & v_jx_2\\
\end{matrix}\right|$$
Note that
$$
\begin{array}{lcll}
\displaystyle\frac{\partial g}{\partial v_j} &=& 
\displaystyle\lambda_1\frac{\partial g^\prime}{\partial v_j} 
& (2\leq j\leq n), \text{ and} \\
\displaystyle\frac{\partial g}{\partial u_{ij}} &=& 
\displaystyle\lambda_1\frac{\partial g^\prime}{\partial u_{ij}} 
& (2\leq j\leq i\leq n) \\
\end{array}
$$
where $g^\prime$ is the determinant of 
the $(n-1)\times (n-1)$ matrix obtained from 
$x_1Y_1+x_2Y_2+x_3E_n$ by removing the first row and the first column
minus $x_3^{n-1}$.
Therefore we have
$$c(v_j)=c(u_{ij})=0 \quad (2\leq j\leq i\leq n)$$ 
since
$\displaystyle\frac{\partial g^\prime}{\partial v_j}$, 
$\displaystyle\frac{\partial g^\prime}{\partial u_{ij}}$ 
($2\leq j\leq i\leq n$) 
are linearly independent by \cite[Lemma 5.2]{Sumi-etal:2011a}.
By \eqref{eqn:1}, we have
\begin{equation}\label{eqn:4}
c(u_{11})\frac{\partial g}{\partial u_{11}}
-\sum_{i=2}^n c(u_{i1})\displaystyle u_{11}x_1^2x_2^{n-i}\mu(2,i-1)=0.
\end{equation}
By expanding at the $n$-th column, we have
$$
\frac{\partial g}{\partial u_{11}}=(v_2+1)x_1x_2^{n-1}+\sum_{i=3}^{n-1} v_{i}x_1x_2^{n-i+1}
\mu(2,i-1)+ x_1(\lambda_n+v_nx_2)\mu(2,n-1)
$$
and then the equation \eqref{eqn:4} implies that
\begin{equation*}
\begin{split}
(c(u_{11})(v_2+1)x_2-&c(u_{21})x_1)x_1x_2^{n-2}+
\sum_{i=3}^{n} (c(u_{11})v_ix_2-c(u_{i1})x_1)x_1x_2^{n-i}\mu(2,i-1) \\
&+(c(u_{11})(\lambda_n+v_nx_2)-c(u_{n1})u_{11}x_1)x_1\mu(2,n-1)=0.
\end{split}
\end{equation*}
By seeing the coefficient of $x_1x_2^{n-1}$, we have $c(u_{11})=0$.
Further, by seeing the coefficients corresponding
to $x_3^s$, $0\leq s\leq n-2$ in the equation \eqref{eqn:4}, 
we have $c(u_{i1})=0$ for $2\leq i\leq n$.
Therefore, we conclude that 
$\displaystyle\frac{\partial g}{\partial v_j}$, $\displaystyle\frac{\partial g}{\partial u_{ij}}$ are linearly independent, which
means that the Jacobian of $g$ is nonzero.

Therefore the set of $(Z_1,\ldots,Z_{m-1})\in \RRR^{n\times p}$
such that
$(Z_{m-1})_{=n}=\zerovec$ and
the polynomial \eqref{eq:char.poly.} is irreducible
is a Zariski open set.
Hence, the condition \eqref{cd5} is replaced with \eqref{eq:char.poly.}.
\qed
\end{proof}

\end{document}